\documentclass[reqno]{amsart}
\usepackage[utf8]{inputenc}
\usepackage{amsmath}
\usepackage{amsfonts}
\usepackage{amssymb}
\usepackage{amsthm}
\usepackage{bbm}
\usepackage{cancel}
\usepackage{color, comment}
\usepackage[colorlinks,citecolor=blue]{hyperref}
\usepackage{tikz}
\usepackage{subcaption}
\usepackage{listofitems}
\usepackage{pgfplots}
\pgfplotsset{compat=1.15}
\usepackage{mathrsfs}
\usetikzlibrary{arrows}

\definecolor{ccqqqq}{rgb}{0.8,0.,0.}
\definecolor{qqqqff}{rgb}{0.,0.,1.}
\definecolor{xfqqff}{rgb}{0.4980392156862745,0.,1.}

\numberwithin{equation}{section}
\newcommand{\eps}{\varepsilon}
\renewcommand{\H}{\mathcal{H}}

\newcommand{\N}{\mathbb{N}}
\renewcommand{\S}{\mathcal{S}}

\newcommand{\R}{\mathbb{R}}

\renewcommand{\div}{\mathrm{div}\,}
\newcommand{\Tan}{\mathrm{Tan}\,}

\newcommand{\1}{\mathbbm{1}}

\newtheorem{theorem}{Theorem}[section]
\newtheorem{lemma}[theorem]{Lemma}
\newtheorem{proposition}[theorem]{Proposition}

\theoremstyle{definition}
\newtheorem{remark}[theorem]{Remark}

\title{Endpoint Fourier restriction and unrectifiability}

\author{Giacomo Del Nin}
\address{Giacomo Del Nin: Mathematics Institute, The University of Warwick, Zeeman Building, CV47HP Coventry, UK.}
\email{\href{mailto:Giacomo.Del-Nin@warwick.ac.uk}{Giacomo.Del-Nin@warwick.ac.uk}}
\author{Andrea Merlo}
\address{Andrea Merlo: Université Paris-Saclay, 307 Rue Michel Magat Bâtiment, 91400 Orsay, France.}
\email{\href{mailto:andrea.merlo@u-psud.fr}{andrea.merlo@u-psud.fr}}

\begin{document}

\maketitle

\begin{abstract}
    We show that if a measure of dimension $s$ on $\R^d$ admits $(p,q)$ Fourier restriction for some endpoint exponents allowed by its dimension, namely $q=\tfrac{s}{d}p'$ for some $p>1$, then it is either absolutely continuous or $1$-purely unrectifiable.
\end{abstract}
{\let\thefootnote\relax}
{\let\thefootnote\relax}

{\let\thefootnote\relax\footnote{The first named author has received funding from the European Research Council (ERC) under the European Union's Horizon 2020 research and innovation programme, grant agreement No 757254 (SINGULARITY). The second named author is supported by the Simons Foundation Wave Project. The second named author would also like to thank prof. Paolo Ciatti for the fruitful discussion on the topic carried out throughout the years.}}

\section{Introduction}
A measure $\mu$ on $\R^d$ is said to admit $(p,q)$ restriction, indicated by $R_\mu(p\to q)$, if there exists a constant $C$ such that for every $f$ in the Schwartz space $\mathcal S(\R^d)$
\begin{equation}\label{eq:restriction}
\|\hat f\|_{L^q(\mu)}\leq C\|f\|_{L^p(\R^d)}.
\end{equation}
Given a measure $\mu$ a broadly open problem is understanding the relation between the admissible exponents $p,q$ in \eqref{eq:restriction}, the dimension of the measure, and its geometric properties.
Assuming that $R_\mu(p\to q)$ holds, it is possible to obtain some necessary conditions on $p$ and $q$ using the so-called Knapp example: if $\mu$ is assumed to have dimension $s$ (without any assumption on its geometry), then $q\leq \tfrac{s}{d}p'$  (see Lemma \ref{lem:necessary_cond_dim}); if we also assume that $\mu$ is $s$-rectifiable of class $C^{1,\alpha}$, then $q\leq\tfrac{s}{d+\alpha(d-s)}p'$ (see Lemma \ref{lem:necessary_cond_rectif}). We will refer to the equality $q=\tfrac{s}{d}p'$ as the \textit{endpoint case}, since it is the extremal case allowed by the dimension.
We are interested in understanding what information on $\mu$ can be inferred from the only assumption that $R_\mu(p\to q)$ holds in the endpoint case $q=\tfrac{s}{d}p'$, for some $p>1$. The following is our main result.

\begin{theorem}[Dicotomy for endpoint estimates]\label{thm:endpoint_implies_unrectif}
Let $\mu$ be a measure on $\R^d$, satisfying $0<\Theta^{*s}(\mu,x)<\infty$ for $\mu$-a.e. $x\in\R^d$. Suppose that $R_\mu(p\to q)$ holds at the endpoint $q=\frac{s}{d}p'$, for some $p>1$. Then
either $q=p'$ and $\mu\ll\mathcal{L}^d$, or $\mu$ is supported on a $1$-purely unrectifiable set.
\end{theorem}

A $1$-purely unrectifiable set is a set that intersects every Lipschitz curve in an $\H^1$-negligible set. Therefore Theorem \ref{thm:endpoint_implies_unrectif} is imposing  strong  geometric  constraints for measures satisfying the endpoint estimate, showing that all non absolutely continuous examples of such measures must be highly singular.

A natural question arises: are there even examples of measures of dimension $s$ (besides $s=d$) satisfying the endpoint restriction for some $q$, to which Theorem \ref{thm:endpoint_implies_unrectif} can be applied? 
If we require Ahlfors-David regularity on $\mu$ then there are a few negative results: as shown by Chen \cite[Proposition 3]{chen},  for $\tfrac{d}{2}\leq s<d$ there are no nontrivial AD-regular measures of dimension $s$, supported on a compact set, admitting a restriction $R_\mu(\tfrac{2d}{2d-s}\to q)$ for some $q$ (note that $p=\tfrac{2d}{2d-s}$ is the endpoint case when $q=2$, but the mentioned result states that the restriction does not hold, for the same $p$, even in the weaker case $q=1$). In the case $q=2$, using convolution powers, Chen and Seeger proved that a measure admitting an endpoint estimate $R_\mu(\tfrac{2d}{2d-s}\to 2)$ can not be AD-regular of dimension $s$ for any $0<s<d$, \cite[Proposition 4.4]{chen-seeger}, and the same proof can be adapted to also exclude the case where $0<\Theta^s_*(\mu,x)\leq \Theta^{*s}(\mu,x)<\infty$ for $\mu$-a.e. $x$.

On the other hand, if we drop the requirement of lower regularity, Chen and Seeger provide, for any $d\geq 1$ and $j\in \N$, an upper regular (but not lower regular) measure of dimension $s=\tfrac{d}{j}$ for which the endpoint estimate holds for $q=2$ \cite[Theorem A]{chen-seeger} and this provides an example where Theorem \ref{thm:endpoint_implies_unrectif} applies. 
Later \L{}aba and Wang extended the previous result to all $s\in (0,d)$, not necessarily of the form $\tfrac{d}{j}$, but their examples satisfy the restriction up to the endpoint excluded \cite[Theorem~2]{laba-wang}.
We also mention a result due to Bilz, who constructs a “universally bad” compact set of Hausdorff dimension $d$, such that no restriction is possible (except for $p=1$) for a measure with support in it \cite[Corollary 2]{bilz}. 


We conclude referring the reader to the survey \cite{laba} for an overview on the restriction problem for fractal measures. We just mention here the results on the restriction for general measures due to Mockenhaupt \cite{mockenhaupt}, Mitsis \cite{mitsis} and Bak-Seeger \cite{bak-seeger}, and the proof of sharpness of the previous results by Hambrook-\L{}aba \cite{hambrook-laba}. 

\section{Proofs}

\subsection{Notation}
 The Fourier transform of a function $f\in\S(\R^d)$ is defined by
\[
\hat f(\xi):=\int e^{-2\pi i\xi\cdot x}f(x)dx.
\]
 We define the map $T_{x,r}(y)=\tfrac{y-x}{r}$, and we denote by $T_{x,r}\mu$ the pushforward of $\mu$ under $T_{x,r}$, namely $T_{x,r}\mu(A):=\mu(x+rA)$ for every Borel set $A$. Given a measure $\mu$ and a point $x$, the space of \textit{tangent measures}  $\Tan(\mu,x)$, introduced in \cite{preiss}, consists of all the possible limits (in the weak* sense of measures) of $c_iT_{x,r_i}\mu$, for some sequence of positive real numbers $c_i$ and some sequence of radii $r_i\to 0$. 
 We denote by $\H^s$ the $s$-dimensional Hausdorff measure, and by $\Theta^{*s}(\mu,x)$ and $\Theta^s_*(\mu,x)$ the $s$-dimensional upper and lower densities of $\mu$ at the point $x$ \cite[2.10.19]{federer}.
 
 When $s$ is an integer, a measure $\mu$ is $s$-rectifiable of class $C^{1,\alpha}$ if $\mu\ll \H^s$ and if it is supported on a set that can be covered by countably many (rotated) graphs of $C^{1,\alpha}$ maps from $\R^s$ to $\R^{d-s}$. A set $E\subset\R^d$ is $s$-purely unrectifiable if, for every Lipschitz map $\phi:\R^s\to \R^d$, $\H^s(E\cap \phi(\R^s))=0$.

\subsection{Preliminary facts}
\begin{lemma}[Necessary conditions and dimension]\label{lem:necessary_cond_dim} 
Let $\mu$ be a Radon measure on $\R^d$ such
that $R_\mu(p\to q)$ holds. Then there exists a constant $M$ depending only on $d,p,q,C$ such that  $\mu(B(x,r))\leq Mr^{{dq}/{p^\prime}}$ for any $x\in\R^n$ and $r>0$ and in particular $\mu\ll \mathcal{H}^{{dq}/{p^\prime}}$.
Moreover, if $0<\Theta^{*s}(\phi,x)<\infty$ on a set of positive $\mu$-measure, then $q\leq \tfrac{s}{d}p^\prime$.
\end{lemma}

\begin{proof}
Since $\mu$ satisfies the restriction inequality with exponents $p$ and $q$ we know that for any $t>0$ we have
\begin{equation}
    \begin{split}
        \int \lvert\widehat{e^{-\pi t \lvert \boldsymbol{\cdot} \rvert^2}}(\zeta-x_0)\rvert^q d\mu(\zeta)=&\int \lvert\widehat{e^{-\pi  t \lvert \boldsymbol{\cdot} \rvert^2+2\pi i \boldsymbol{\cdot}\cdot x_0}}(\zeta)\rvert^q d\mu(\zeta)\\
        \leq& C^q \Big(\int e^{-\pi p t \lvert x \rvert^2}d\mathcal{L}^n(x)\Big)^\frac{q}{p}.
        \nonumber
    \end{split}
\end{equation}
The Fourier transform of the function $e^{-\pi t \lvert x \rvert^2}$ is  $t^{-d/2}e^{-\pi \lvert \zeta \rvert^2/t}$. This implies that
$$\int t^{-dq/2}e^{-\pi q \lvert \zeta-x_0 \rvert^2/t} d\mu(\zeta)\leq C^q \Big(\int e^{-\pi p t \lvert x \rvert^2}d\mathcal{L}^d(x)\Big)^\frac{q}{p}=C^q(pt)^{-\frac{dq}{2p}}.$$
Finally, rearranging the above inequality we conclude that
\begin{equation}
    \begin{split}
e^{-1}\mu(B(x_0,(\pi q)^{-1/2}t^{1/2}))\leq C^qt^{dq/2}(pt)^{-\frac{dq}{2p}}=C^qp^{-\frac{dq}{2p}}t^\frac{dq}{2p^\prime}.
\label{eq:num1}
\end{split}
\end{equation}
The chain of inequalities in \eqref{eq:num1} yields a constant $M=M(d,p,q,C)$ such that $\mu(B(x,r))\leq M r^{{dq}/{p^\prime}}$ for any $x\in \R^d$ and any $r>0$. The fact that $\mu$ is absolutely continuous with respect to $\mathcal{H}^{{dq}/{p^\prime}}$ is an immediate consequence of \cite[2.10.19(1)]{federer}. The last assertion of the theorem follows because if the density assumption holds at $x$ then for an infinitesimal sequence of radii we also have $\mu(B(x,r))\gtrsim r^s$.
\end{proof}

\begin{remark}
Note that Lemma \ref{lem:necessary_cond_dim} implies that if $p=q^\prime$ and $p>1$, then $\mu\ll\mathcal{L}^d$. 
\end{remark}

\begin{lemma}[Necessary conditions and rectifiability]\label{lem:necessary_cond_rectif}
Let $s$ be an integer and $\mu$ be a measure on $\R^d$ which is $s$-rectifiable of class $C^{1,\alpha}$. Suppose that $R_\mu(p\to q)$ holds. Then $q\leq \frac{s}{d+\alpha(d-s)}p'$. 
\end{lemma}

\begin{proof}
We sketch the proof, since it is a slight variation of the so-called Knapp example, which exploits the fact that the restriction estimate \eqref{eq:restriction} is equivalent to the following extension estimate:
\[
\|\widehat{g\mu}\|_{L^{p'}(\R^d)}\leq C\|g\|_{L^{q'}(\mu)}\qquad\text{for every $g\in C^\infty_c(\R^d)$},
\]
where
\begin{equation}\label{eq:extension}
\widehat{g\mu}(x):=\int e^{-2\pi i x\cdot \xi}g(\xi)d\mu(\xi).
\end{equation}
We fix a non negative bump function $\phi$ which is $1$ in $B(0,\tfrac12)$ and with support in $B(0,1)$. Given $\delta>0$ and $\xi_0\in \R^d$ such that $0<\Theta^s_*(\mu,\xi_0)\leq \Theta^{*s}(\mu,\xi_0)<\infty$, we define $g(\xi):=\phi(\tfrac{\xi-\xi_0}{\delta})$. Using that $\mu(B(\xi_0,\delta))\approx \delta^s$ for small $\delta$, it is readily seen that $\|g\|_{L^{q'}(\mu)}\approx\delta^{s/q'}$ for small $\delta$. On the other hand, since $\mu$ is $s$-rectifiable of class $C^{1,\alpha}$, by Proposition 1.2 and Remark 1.4(iii) in \cite{dni} we know that in the ball $B(\xi_0,\delta)$ the measure is mostly concentrated on a $c\delta^{1+\alpha}$-neighbourhood of an $s$-plane for some constant $c$, that is
\begin{equation}\label{eq:cylinder}
\frac{\mu(B(\xi_0,\delta)\setminus B(\xi_0+V,c\delta^{1+\alpha}))}{\mu(B(\xi_0,\delta))}\to 0 \qquad\text{as $\delta\to 0$},
\end{equation}
where $V$ is a linear $s$-subspace of $\R^d$.
Given $\xi\in B(\xi_0+V,c\delta^{1+\alpha})\cap B(\xi_0,\delta)$, we have that $x\cdot \xi$ is small (so that $e^{-2\pi i x\cdot \xi}$ is of order $1$) for every $x$ in a dual cylinder $B(V^\perp, \tfrac{c_0}{\delta^{1+\alpha}})\cap B(0,\tfrac{c_0}{\delta})$, for some sufficiently small constant $c_0$. Hence, using also \eqref{eq:cylinder}, from \eqref{eq:extension} we see that for small $\delta$
\begin{equation}\label{eq:lower_bound_tubes}
|\widehat{g\mu}(x)|\gtrsim \delta^k\qquad\text{for all $x\in B(V^\perp, \tfrac{c_0}{\delta^{1+\alpha}})\cap B(0,\tfrac{c_0}{\delta})$}.
\end{equation}
From this we obtain $\|\widehat{g\mu}\|_{L^{p'}(\R^n)}\gtrsim \delta^s \left(\tfrac{1}{\delta^{(1+\alpha)(d-s)+s}}\right)^{1/p'}$. The restriction estimate thus implies that for small $\delta$
\[
\delta^s\left(\tfrac{1}{\delta^{(1+\alpha)(d-s)+s}}\right)^{1/p'}\lesssim \delta^{s/q'}
\]
which after a quick computation forces $q\leq \frac{s}{d+\alpha(n-s)}p'$.
\end{proof}

\begin{remark}\label{rmk:second_condition}
Assuming also positive lower density of $\mu$, it is possible to prove a second necessary condition, namely $p\leq \tfrac{2d}{2d-s}$. This can be done putting together $N\approx {\delta^{-s}}$ Knapp examples with disjoint supports, and random signs $\eps_i\in\{-1,+1\}$, that is testing \eqref{eq:extension} with 
\[
g(\xi):=\sum_{i=1}^N \eps_i g_i(\xi),\qquad g_i(\xi):=\phi\left(\frac{\xi-\xi_i}{\delta}\right).
\]
In this case $\|g\|_{L^{q'}(\mu)}\approx 1$, while Khintchine's inequality yields the existence of a choice of $\eps_i$ such that
\[
\|\widehat{g\mu}\|_{L^{p'}(\R^d)}=\Big\|\sum_{i=1}^N\eps_i\widehat{g_i\mu}\Big\|_{L^{p'}(\R^d)}\gtrsim \Big\|\Big(\sum_{i=1}^N |\widehat{g_i\mu}|^2\Big)^{\frac12}\Big\|_{L^{p'}(\R^d)}.
\]
Similarly to \eqref{eq:lower_bound_tubes} we have that $\widehat{g_i\mu}(x)\gtrsim \delta^s$ for all $x\in B(0,\tfrac{c_0}{\delta})$, hence
\[
\Big\|\Big(\sum_{i=1}^N |\widehat{g_i\mu}|^2\Big)^{\frac12}\Big\|_{L^{p'}(\R^d)}\gtrsim \|(N\delta^{2s}\1_{B(0,\frac{c_0}{\delta})})^\frac12 \|_{L^{p'}(\R^d)}\approx N^\frac12 \delta^s \frac{1}{\delta^{d/p'}}\approx \delta^{\frac{s}{2}-\frac{d}{p'}}.
\]
The restriction inequality then implies that $1\gtrsim \delta^{\frac{s}{2}-\frac{d}{p'}}$ for small $\delta$, which yields the conclusion.
\end{remark}
\begin{remark}
In the case $s=d-1$ and $\alpha=1$ (that is when $\mu$ is $(d-1)$-rectifiable of class $C^{1,1}$) the conditions given by Lemma \ref{lem:necessary_cond_rectif} and  Remark \ref{rmk:second_condition} coincide (except for the endpoint) with the optimal range conjectured for the restriction on the sphere, one of the most important cases \cite{tao}.
\end{remark}

\begin{lemma}[Stability of restriction under blowup]\label{lem:restriction_blowup}
Suppose $\mu$ is a Radon measure for which $R(p\to sp^\prime/d)$ holds for some $s\in[0,d]$. If there exists an $x\in\R^d$ and an infinitesimal sequence $r_i$ such that $r_i^{-s}T_{x,r_i}\mu\rightharpoonup \nu$, then $R(p\to sp^\prime/d)$ holds for $\nu$ with the same constant.
\end{lemma}

\begin{proof}
Let us put $q=\tfrac{s}{d}p'$. In order to simplify the notations, in the following we let $\mu_i:=r_i^{-s}T_{x,r_i}\mu$. For any 
$f\in \mathcal{S}(\R^d)$, using the scaling properties of the Fourier transform (indicated here also with $\mathcal{F}$) and a change of variables, we have
\begin{equation}
\begin{split}
    \|\hat f\|_{L^q(\mu_i)}^q=\int \lvert \hat f(y)\rvert^{q} d\mu_i(y)&= r_i^{-s}\int \lvert \hat f(y)\rvert^{q}dT_{x,r_i}\mu(y)=r_i^{-s}\int \Big\lvert \hat f\Big(\frac{y-x}{r_i}\Big)\Big\rvert^{q} d\mu(y)\\
    & =r_i^{-s}\int \left| r_i^d\mathcal{F}\left( f(r_i z)e^{2\pi i x\cdot z}\right)(y)\right|^{q} d\mu(y)\\
    & \leq C^{q}r_i^{-s+dq}\left(\int \left| f(r_i z)e^{2\pi i x\cdot z}\right|^{p} dz\right)^{\frac{q}{p}}\\
    & =C^{q}r_i^{-s+dq-d\frac{q}{p}}\| f\|_{L^{p}(\R^d)}^{q}=C^{q}\| f\|_{L^{p}(\R^d)}^{q}.
    \nonumber
\end{split}
\end{equation}
Thanks to the regularity of $f$, from the definition of $\nu$ we deduce that
$\int \lvert \hat f(y)\rvert^p d\mu_i(y)$ converges to $\int \lvert \hat f(y)\rvert^p d\nu(y)$, concluding the proof.
\end{proof}

Next we see that, for measures which are the tensor product of Lebesgue with some other measure, the only possibility for the restriction is $p=q'$, as in the case of the Lebesgue measure itself.
\begin{lemma}[Restriction for tensor measures]\label{lem:restriction_tensor}
Let us consider a measure of the form $\eta=\theta\otimes \H^k\llcorner V$, where $V$ is a $k$-dimensional subspace of $\R^d$, and $\theta$ is a measure on $V^\perp$. Suppose that $R_\mu(p\to q)$ holds for some $p,q\in [1,\infty]$. Then either $k=0$ or $p=q'$. In any case $\theta$ admits, within $V^\perp$, the same restriction estimate $R_\theta(p\to q)$.
\end{lemma}

\begin{proof}
We consider Schwartz functions of the form $f(x)=f_1(x_1)f_2(x_2)$ where $x=x_1+x_2$, $x_1\in V$, $x_2\in V^\perp$, and $f_1\in\S(V)$, $f_2\in\S(V^\perp)$. Then $\hat f(x)=\hat f_1(x_1)\hat f_2(x_2)$. Using that $\|f_1f_2\|_{L^p(\nu_1\otimes\nu_2)}=\|f_1\|_{L^p(\nu_1)}\|f_2\|_{L^p(\nu_2)}$, and assuming that $R_\mu(p\to q)$ holds we obtain 
\begin{align*}
\|\hat f_1\|_{L^q(\H^k\llcorner V)}\|\hat f_2\|_{L^q(\theta)}&=\|\hat f_1 \hat f_2\|_{L^q(\mu)}=\|\hat f\|_{L^q(\mu)}\\
&\leq C \|f\|_{L^p(\R^d)}=C\|f_1\|_{L^p(\H^k\llcorner V)} \|f_2\|_{L^p(\H^{d-k}\llcorner V^\perp)}
\end{align*}
which can be rewritten as
\begin{equation}\label{eq:fraction}
\|\hat f_2\|_{L^q(\theta)}\leq C\frac{\|f_1\|_{L^p(\H^k\llcorner V)}}{\|\hat f_1\|_{L^q(\H^k\llcorner V)}} \|f_2\|_{L^p(\H^{d-k}\llcorner V^\perp)}.
\end{equation}
If $q\neq p'$ and $k\neq 0$, thanks to the absence of restriction estimates from the measure $\H^k\llcorner V$ to itself we reach a contradiction: we can fix any $f_2$ for which both norms in the inequality are nonzero (for instance a Gaussian) and then choose $f_1$ that makes the fraction sufficiently small, thus contradicting the inequality. This forces either $k=0$ or $p=q'$. Once the fraction in \eqref{eq:fraction} is bounded, it is clear that $\theta$ (as a measure on $V^\perp$) admits the same restriction estimate $R_\theta(p\to q)$.
\end{proof}

\begin{remark}
As a consequence of Lemma \ref{lem:restriction_blowup} and Lemma \ref{lem:restriction_tensor} we immediately obtain a weaker version of Theorem \ref{thm:endpoint_implies_unrectif}: if $\mu$ is $s$-dimensional, $s<d$, and admits restriction for the endpoint $q=\tfrac{s}{d}p'$, $p>1$, then $\mu$ is $s$-purely unrectifiable. Indeed, if this were not the case, then we could find a flat tangent measure, that is a point $x$ and a sequence $r_i^{-s}T_{x,r_i}\mu$ converging to a measure $\nu$ which is a nontrivial multiple of $\H^s$ restricted to an $s$-plane. By Lemma \ref{lem:restriction_blowup}, in the endpoint case the restriction passes to $\nu$. However we can now apply Lemma \ref{lem:restriction_tensor} with $\theta=\delta_0$: either $p=q'$ (and then $s=d$) or $R_\theta(p\to q)$ holds for $\theta=\delta_0$ on a nontrivial subspace, which forces $p=1$.
\end{remark}

\subsection{Decomposability bundle}
In the proof of Theorem \ref{thm:endpoint_implies_unrectif} we will use the \textit{decomposability bundle} $V(\mu,x)$, defined for every Radon measure $\mu$ in $\R^d$ and introduced in \cite{alberti-marchese}. $V(\mu,x)$ is  a linear subspace of $\R^d$ defined for $\mu$-a.e. $x$, and contains all the directions along which $\mu$ is decomposable, near $x$, in $1$-rectifiable measures. Originally introduced to study the directions of differentiability of Lipschitz functions, the decomposability bundle is characterized by the following property\footnote{The referenced result is stated in terms of boundaryless $1$-currents but, as observed in \cite{alberti-marchese}, it is equivalent to consider measures $\sigma$ with values in $\R^n$, for which the boundaryless requirement becomes $\div \sigma=0$.} \cite[Theorem 6.4]{alberti-marchese}: for $\mu$-a.e. $x\in\R^d$, $v$ belongs to $V(\mu,x)$ if and only if there exists an $\R^d$-valued measure $\sigma$ such that $\div \sigma=0$ in the distributional sense and
\begin{equation}\label{eq:decomp_bundle_def}
\lim_{r\to 0}\frac{|\sigma-v\mu|(B(x,r))}{\mu(B(x,r))}=0.
\end{equation}
We will use that $V(\mu,x)=\{0\}$ for $\mu$-a.e. $x$ if and only if $\mu$ is supported on a $1$-purely unrectifiable set \cite[Proposition~2.9(iv)]{alberti-marchese}.

The proof of Theorem \ref{thm:endpoint_implies_unrectif} relies on the following proposition, connecting the decomposability bundle $V(\mu,x)$ to the structure of tangent measures of $\mu$ at $x$. Up to the author's knowledge this result is not present in the literature, even though it is well-known within the community, see for instance the strategy proposed in \cite[p. 643]{GdPR}. 

\begin{proposition}[Tangent measures and decomposability bundle]\label{prop:tangent_bundle}
Let $\mu$ be a Radon measure on $\R^d$. Then for $\mu$-a.e. $x\in\R^d$, any tangent measure $\eta\in \Tan(\mu,x)$ is $V(\mu,x)$-invariant, namely there exists a Radon measure $\theta$ on $V(\mu,x)^\perp$ such that $\eta=\theta\otimes \H^k\llcorner V(\mu,x)$, where $k=\dim V(\mu,x)$.
\end{proposition}

\begin{proof}
Let us consider  a point $x$ in the support of $\mu$, and a tangent measure $\eta\in\Tan(\mu,x)$ given by
\[
\eta=\lim_{i\to\infty}c_i T_{x,r_i}\mu.
\]
Let us take $v\in V(\mu,x)$ and $\sigma$ an $\R^d$-valued measure with $\div \sigma=0$ for which \eqref{eq:decomp_bundle_def} holds. 
Then
\begin{align*}
0 =\lim_{i\to\infty}\frac{|\sigma-v\mu|(B(x,\rho r_i))}{\mu(B(x,\rho r_i))}& =\lim_{i\to\infty}\frac{|c_i T_{x,r_i}\sigma-v c_iT_{x,r_i}\mu|(B(x,\rho))}{c_iT_{x,r_i}\mu(B(x,\rho))}\\
&\geq\limsup_{i\to\infty} \frac{|c_i T_{x,r_i}\sigma-v c_iT_{x,r_i}\mu|(B(x,\rho))}{\eta(B(x,\rho))}.
\end{align*}
We know that $vc_iT_{x,r_i}\mu$ converges weakly to $v\eta$, thus from the above computation we infer that the same is true for $c_iT_{x,r_i}\sigma$. Since the divergence constraint is invariant under scalings and passes to the limit we obtain $\div(v\eta)=0$. From the constancy of $v$ we infer that $\eta$ is invariant in direction $v$, namely that $T_{v,1}\eta=\eta$. As a consequence, $T_{v,1}\eta=\eta$ for every $v\in V(\mu,x)$. By disintegration and Haar's theorem this yields the existence of a measure $\theta$ on $V(\mu,x)^\perp$ such that $\eta=\theta\otimes \H^k\llcorner V(\mu,x)$.
\end{proof}

\subsection{Proof of Theorem \ref{thm:endpoint_implies_unrectif}}
Let us fix any point $x$ where $0<\Theta^{*s}(\mu,x)<\infty$ and where Proposition \ref{prop:tangent_bundle} applies. Both conditions are satisfied on a set of full $\mu$-measure. From the density assumption there exists a sequence of radii $r_i\to 0$, $i\in \N$, such that
\[
M^{-1}\leq r_i^{-s}\mu(B(x,r_i))\leq M
\]
for some constant $M>0$ and every $i\in\N$. From the sequence $r_i^{-s}T_{x,r_i}\mu$ we can thus extract a subsequence converging to some non zero  tangent measure $\eta\in \Tan(\mu,x)$. Since $p=\frac{s}{d}q'$, by Lemma \ref{lem:restriction_blowup} the restriction estimate $R_\mu(p\to q)$ passes to the tangent measure $\eta$. By Proposition \ref{prop:tangent_bundle}, $\eta$ is $V(\mu,x)$-invariant, that is of the form $\theta\otimes \H^k\llcorner V(\mu,x)$, with $\theta$ measure on $V(\mu,x)^\perp$. By Lemma \ref{lem:restriction_tensor} there are only two possibilities: either $p=q'$, which forces $s=d$ and then $\mu\ll \mathcal{L}^d$; or $p\neq q'$, and then necessarily $k=0$. If the second condition holds it means that $V(\mu,x)=\{0\}$ at $\mu$-a.e. $x$, which by \cite[Proposition~2.9(iv)]{alberti-marchese} is equivalent to $\mu$ being supported on a $1$-purely unrectifiable set. This concludes the proof.\hfill\qedsymbol

\bibliographystyle{plain}
\bibliography{ref.bib}




 \end{document}